\documentclass[12pt]{article}
\usepackage{amsmath}
\usepackage{amssymb}
\usepackage{amsmath,amssymb,amsbsy,amsfonts,amsthm,latexsym,
amsopn,amstext,amsxtra,euscript,amscd}

\begin{document}

\def\A{\mathbb{A}}
\def\B{\mathbf{B}}
\def \C{\mathbb{C}}
\def \F{\mathbb{F}}
\def \K{\mathbb{K}}

\def \Z{\mathbb{Z}}
\def \P{\mathbb{P}}
\def \R{\mathbb{R}}
\def \Q{\mathbb{Q}}
\def \N{\mathbb{N}}
\def \Z{\mathbb{Z}}

\def\B{\mathcal B}
\def\e{\varepsilon}

\def\cA{{\mathcal A}}
\def\cB{{\mathcal B}}
\def\cC{{\mathcal C}}
\def\cD{{\mathcal D}}
\def\cE{{\mathcal E}}
\def\cF{{\mathcal F}}
\def\cG{{\mathcal G}}
\def\cH{{\mathcal H}}
\def\cI{{\mathcal I}}
\def\cJ{{\mathcal J}}
\def\cK{{\mathcal K}}
\def\cL{{\mathcal L}}
\def\cM{{\mathcal M}}
\def\cN{{\mathcal N}}
\def\cO{{\mathcal O}}
\def\cP{{\mathcal P}}
\def\cQ{{\mathcal Q}}
\def\cR{{\mathcal R}}
\def\cS{{\mathcal S}}
\def\cT{{\mathcal T}}
\def\cU{{\mathcal U}}
\def\cV{{\mathcal V}}
\def\cW{{\mathcal W}}
\def\cX{{\mathcal X}}
\def\cY{{\mathcal Y}}
\def\cZ{{\mathcal Z}}

\def\f{\frac{|\A||B|}{|G|}}
\def\AB{|\A\cap B|}
\def \Fq{\F_q}
\def \Fqn{\F_{q^n}}

\def\({\left(}
\def\){\right)}
\def\fl#1{\left\lfloor#1\right\rfloor}
\def\rf#1{\left\lceil#1\right\rceil}
\def\Res{{\mathrm{Res}}}

\newcommand{\comm}[1]{\marginpar{
\vskip-\baselineskip \raggedright\footnotesize
\itshape\hrule\smallskip#1\par\smallskip\hrule}}

\newtheorem{lem}{Lemma}
\newtheorem{lemma}[lem]{Lemma}
\newtheorem{prop}{Proposition}
\newtheorem{proposition}[prop]{Proposition }
\newtheorem{thm}{Theorem}
\newtheorem{theorem}[thm]{Theorem}
\newtheorem{cor}{Corollary}
\newtheorem{corollary}[cor]{Corollary}
\newtheorem{prob}{Problem}
\newtheorem{problem}[prob]{Problem}
\newtheorem{ques}{Question}
\newtheorem{question}[ques]{Question}
\newtheorem{rem}{Remark}

\title{On the congruence $x^{x}\equiv \lambda\pmod p$}

\author{
{\sc J. Cilleruelo} and {\sc M.~Z.~Garaev} }

\date{}

\maketitle

\begin{abstract}
In the present paper we obtain several new results related to the
problem of upper bound estimates for the number of solutions of the
congruence
$$
x^{x}\equiv \lambda\pmod p;\quad x\in \mathbb{N},\quad x\le p-1,
$$
where $p$ is a large prime number, $\lambda$ is an integer corpime
to $p$. Our arguments are based on recent estimates of trigonometric
sums over subgroups due to Shkredov and Shteinikov.
\end{abstract}

\section{Introduction}

For a prime $p$ and an integer $\lambda$ let $J(p; \lambda)$ be the
number of solutions of the congruence
\begin{equation}
\label{eqn:MainCongr}
x^{x}\equiv \lambda\pmod p;\quad x\in \mathbb{N},\quad x\le p-1.
\end{equation}
Note that the period of the function $x^x$ modulo $p$ is $p(p-1)$,
which is larger than the range in congruence~\eqref{eqn:MainCongr}.

From the works of Crocker~\cite{Cr} and Somer~\cite{Som} it is known
that there are at least $\lfloor(p-1)/2\rfloor$ and at most $3p/4
+p^{1/2+o(1)}$ incongruent values of $x^x\pmod p$ when $1\le x\le
p-1$. There are several conjectures in~\cite{HM} related to this
function.

New approaches to study $J(p;\lambda)$ were given by Balog, Broughan
and Shparlinski, see~\cite{BBSh} and~\cite{BBSh2}. In the special
case $\lambda=1$ it was shown in~\cite{BBSh} that $J(p; 1)<p^{1/3 +
o(1)}$. This estimate was slightly improved in our work~\cite{JG0}
to the bound $J(p; 1)\ll p^{1/3 - c}$ for some absolute constant
$c>0$. Note that the method of~\cite{JG0} applies for a more general
exponential congruences, however, the constant $c$ there becomes too
small. In the present paper we use a different approach and prove
the following results.

\begin{theorem}
\label{thm:Main1} The number $J(p;1)$ of solutions of the congruence
\begin{equation}
\label{eqn:Th1}
x^{x}\equiv 1\pmod p;\quad x\in \mathbb{N},\quad x\le p-1,
\end{equation}
satisfies $J(p;1)\lesssim p^{27/82}$.
\end{theorem}
Here and below we use the notation  $A\lesssim B$ to denote that
$A<Bp^{o(1)}$; that is, for any $\varepsilon>0$ there exists
$c=c(\varepsilon)>0$ such that $A < c B p^{\varepsilon}$. As usual,
${\rm ord}\,\lambda$ denotes the multiplicative order of $\lambda$,
that is, the smallest positive integer $t$ such that
$\lambda^t\equiv 1 \pmod p.$ We recall that ${\rm ord}\, \lambda |
p-1$.

\begin{theorem}
\label{thm:Main2} Uniformly over $t|p-1$, we have, as $p\to\infty$,
\begin{equation}
\label{eqn:Th2}
\sum_{\substack{1\le \lambda\le p-1\\ {\rm ord} \,\lambda =t}}J(p; \lambda)\lesssim t+p^{1/3}t^{1/2}.
\end{equation}
\end{theorem}

In the range $t<p^{1/3}$ our Theorem~\ref{thm:Main2} improves some
results of the aforementioned works~\cite{BBSh} and~\cite{BBSh2}.
Note that in the case $t=1$ the estimate of Theorem~\ref{thm:Main1}
is stronger. In fact, following the argument that we use in the
proof of Theorem ~\ref{thm:Main1} it is posible to improve
Theorem~\ref{thm:Main2} in specific small ranges of $t$.

Let now $I(p)$ denote the number of solutions of the congruence
$$
x^x\equiv y^y\pmod p; \quad x\in \mathbb{N}, \quad y\in \mathbb{N},\quad x\le p-1, \quad y\le p-1.
$$
There is the following relationship between $I(p)$ and
$J(p;\lambda)$:
$$
I(p)=\sum_{\lambda=1}^{p-1}J(p;\lambda)^2.
$$
We modify one of the arguments of~\cite{BBSh} and obtain the
following refinement on ~\cite[Theorem 8]{BBSh}.
\begin{theorem}
\label{thm:Main3} We have, as $p\to\infty,$
\begin{equation}
\label{eqn:Th3}
I(p)\lesssim p^{23/12}.
\end{equation}
\end{theorem}

In order to prove our results, we first reduce the problem to
estimates of exponential sums over subgroups. In the proof of
Theorem~\ref{thm:Main1} we use Shteinikov's result from~\cite{Sht},
while in the proof of Theorem~\ref{thm:Main2} we use Shkredov's
result from~\cite{Shk} (see, Lemma~\ref{lem:Sht} and
Lemma~\ref{lem:Shk} below).

In what follows, $\F_p$ is the field of residue classes modulo $p$.
The elements of $\F_p$ we associate with their concrete
representatives from $\{0,1,\ldots,p-1\}$. For an integer $m$
coprime to $p$ by $m^{*}$ we denote the smallest positive integer
such that $m^{*}m\equiv 1\pmod p$. We also use the abbreviation
$$
e_p(z)=e^{2\pi i z/p}.
$$

{\bf Acknowledgement}. J. Cilleruelo was supported by the grants MTM
2011-22851 of MICINN and ICMAT Severo Ochoa project SEV-2011-0087.
M. Z. Garaev was supported by the sabbatical grant from PASPA-DGAPA-
UNAM.

\section{Lemmas}

\begin{lemma}
\label{lem:CongrMoment} Let
$$
\lambda\not\equiv 0\pmod p,\quad n\in\N,\quad 1\le M\le p.
$$
Then for any fixed constant $k\in \N$ the number $J$ of solutions of
the congruence
$$
x^{n}\equiv \lambda \pmod p,\quad x\in\N,\quad x\le M,
$$
satisfies
$$
J\lesssim \Bigl(1+\frac{M}{p^{1/k}}\Bigr)n^{1/k}.
$$
In particular, if $n=dt<p$ and $M=p/d$, then we have the bound
$$
J\lesssim \Bigl(d^{1/k}+\Bigl(\frac{p}{d}\Bigr)^{1-1/k}\Bigr) t^{1/k}.
$$
\end{lemma}

\begin{proof} We have
$$
J^k\lesssim \#\{(x_1,\ldots,x_k)\in \N^k\cap [1, M]^k;\quad (x_1\ldots x_k)^n\equiv \lambda^k\pmod p\}.
$$
Since for a given integer $\mu$ the congruence
$$
X^n\equiv \mu\pmod p,\quad X\in\N,
\quad X\le p,
$$ has at most $n$
solutions, there exists a positive integer $\lambda_0<p$ such that
$$
J^k\lesssim nJ_1,
$$
where $J_1$ is the number of solutions of the congruence
$$
x_1\ldots x_k\equiv \lambda_0\pmod p;\quad (x_1,\ldots,x_k)\in \N^k\cap [1, M]^k.
$$
It follows that
$$
x_1\ldots x_k=\lambda_0+py;\quad (x_1,\ldots,x_k)\in \N^k\cap [1, M]^k,\quad y\in\Z.
$$
Since the left hand side of this equation does not exceed $M^k$, we
get that $|y|\le M^k/p$. Hence, for some fixed $y_0$ we have
$$
J_1\lesssim \Bigl(1+\frac{M^k}{p}\bigr)J_2,
$$
where $J_2$ is the number of solutions of the equation
$$
x_1\ldots x_k=\lambda_0+py_0;\quad (x_1,\ldots,x_k)\in \N^k\cap [1,
M]^k.
$$
Hence, from the bound for the divisor function it follows that
$J_2\lesssim 1$. Thus,
$$
J^k\lesssim \Bigl(1+\frac{M^k}{p}\bigr)n.
$$
and the result follows.
\end{proof}

Let $H_d$ be the subgroup of $\F_p^*=\F_p\setminus\{0\}$ of order
$d$. From the classical estimates for exponential sums over
subgroups it is known that
$$
\Bigl|\sum_{h\in H_d}e_p(ah)\Bigr|\le p^{1/2}.
$$
For a wide range of $d$ this bound has been improved in a serious of
works. Here, we need the results due to Shteinikov~\cite{Sht} (see
Lemma~\ref{lem:Sht} below) and Shkredov~\cite{Shk} (see
Lemma~\ref{lem:Shk} below). They will be used  in the proof of
Theorem~\ref{thm:Main1} and Theorem~\ref{thm:Main2}, respectively.
\begin{lemma}
\label{lem:Sht} Let $H_d$ be the subgroup of $\F_p^*$ of order
$d<p^{1/2}$. Then for any integer $a\not\equiv 0\pmod p$ the
following bound holds:
$$
\Bigl|\sum_{h\in H_d}e_p(ah)\Bigr|\lesssim p^{1/18}d^{101/126}.
$$
\end{lemma}

\begin{lemma} \label{lem:Shk} Let $H_d$ be the subgroup of $\F_p^*$
of order $d<p^{2/3}$. Then for any integer $a\not\equiv 0\pmod p$
the following bound holds:
$$
\Bigl|\sum_{h\in H_d}e_p(ah)\Bigr|\lesssim p^{1/6}d^{1/2}.
$$
\end{lemma}

The following two results are due to Balog, Broughan and Shparlinski
from~\cite{BBSh} and~\cite{BBSh2}.
\begin{lemma}
\label{lem:BBSh} Uniformly over $t|p-1$, we have, as $p\to\infty$,
$$
\sum_{\substack{1\le \lambda\le p-1\\ {\rm ord} \,\lambda =t}}J(p; \lambda)\lesssim t+p^{1/2}.
$$
\end{lemma}

\begin{lemma}
\label{lem:BBSh1} Uniformly over $t|p-1$ and all integers $\lambda$
with $\gcd(\lambda,p)=1$ and ${\rm ord}\,\lambda=t$, we have, as
$p\to\infty$,
$$
J(p;\lambda)\lesssim pt^{-1/12}.
$$
\end{lemma}

We also need the following lemma.
\begin{lemma}\label{lem:easy}
Let $a,x$ be positive integers and let $d=\gcd(x,p-1)$. Then
$a^d\equiv 1\pmod p.$
\end{lemma}
This lemma is well-known and the proof is simple. Indeed, if ${\rm
ind}\, a$ is indice of $a$ with respect to some primitive root $g$
modulo $p$, then,
$$
x\cdot {\rm ind}\, a\equiv 0\pmod{(p-1)}.
$$
Therefore, $d \cdot {\rm ind}\, a \equiv 0\pmod{(p-1)}$, whence
$a^d\equiv 1 \pmod p$.

The following lemma is also well-known; see, for example, exercise
and solutions to chapter 3 in Vinogradov's book~\cite{Vin} for even
a more general statement.

\begin{lemma}\label{lem:Vin}
For any integers $U$ and $V>U$ the following bound holds:
$$
\sum_{a=1}^{p-1}\Bigl|\sum_{z=U}^V e_p(az)\Bigr| \lesssim p.
$$
\end{lemma}

\section{Proof of Theorem~\ref{thm:Main1}}

We have
$$
J(p;1)=\sum_{d|p-1}J_d',
$$
where $J_d'$ is the number of solutions of~\eqref{eqn:Th1} with
$\gcd(x,p-1)=d.$ It then follows  by Lemma \ref{lem:easy} that
$$
J(p;1)\le \sum_{d|p-1}J_d,
$$
where $J_d$ is the number of solutions of the congruence
$$
z^d\equiv (d^d)^{*}\pmod p,\quad z\in\N, \quad z\le (p-1)/d.
$$
We have therefore,
$$
J(p;1)\le R_1+R_2+R_3+ \sum_{\substack{d|p-1\\ d < p^{3/7}}}J_d,
$$
where
$$
R_1=\sum_{\substack{d|p-1\\ d > p^{5/7}}}J_d;\quad
R_2=\sum_{\substack{d|p-1\\ p^{4/7}<d< p^{5/7}}}J_d; \quad R_3=\sum_{\substack{d|p-1\\ p^{3/7}<d\le p^{4/7}}}J_d.
$$
The trivial estimate $J_d\le p/d$ implies that
$$
R_1\lesssim \sum_{\substack{d|p-1\\ d > p^{5/7}}}\frac{p}{d}\lesssim \sum_{d|p-1}p^{2/7}\lesssim p^{2/7}.
$$
To estimate $R_2$ we use Lemma~\ref{lem:CongrMoment} with $k=3$ and
get
\begin{equation*}
\begin{split}
R_2=\sum_{\substack{d|p-1\\ p^{4/7}<d< p^{5/7}}}J_d \lesssim
\sum_{\substack{d|p-1\\ p^{4/7}<d< p^{5/7}}} (d^{1/3}+(p/d)^{2/3})\lesssim \sum_{d|p-1} p^{2/7}\lesssim p^{2/7}.
\end{split}
\end{equation*}
To estimate $R_3$ we use Lemma~\ref{lem:CongrMoment} with $k=2$ and
get
\begin{equation*}
\begin{split}
R_3=\sum_{\substack{d|p-1\\ p^{3/7}<d< p^{4/7}}}J_d \lesssim
\sum_{\substack{d|p-1\\ p^{3/7}<d< p^{4/7}}} (d^{1/2}+(p/d)^{1/2})\lesssim \sum_{d|p-1} p^{2/7}\lesssim p^{2/7}.
\end{split}
\end{equation*}
Thus,
$$
J(p;1)\lesssim p^{2/7}+\sum_{\substack{d|p-1\\ d < p^{3/7}}}J_d.
$$
Hence, there exists  $d|p-1$ with $d<p^{3/7}$ such that
\begin{equation}
\label{eqn:R lesssim Jd}
J(p;1)\lesssim p^{2/7}+J_{d}.
\end{equation}
Applying Lemma~\ref{lem:CongrMoment} with $k=2$, we get
\begin{equation}
\label{eqn: J_d 1}
J_{d}\lesssim d^{1/2}+(p/d)^{1/2}\lesssim (p/d)^{1/2}.
\end{equation}
Let now $H_d$ be the subgroup of $\F_p^*$ of order $d$. We recall
that $J_d$ is the number of solutions of the congruence
$$
(dz)^d\equiv 1\pmod p;\quad z\in \N, \quad z\le (p-1)/d.
$$
Therefore,
$$
J_d= \#\{z\in \N; \quad z\le (p-1)/d,\quad dz\pmod p\in H_d\}.
$$
It then follows that
$$
J_d=\frac{1}{p}\sum_{a=0}^{p-1}\sum_{1\le z\le (p-1)/d}\,\sum_{h\in H_d}e_p(a(dz-h)).
$$
Separating the term corresponding to $a=0$ and using
Lemma~\ref{lem:Sht} for $a\not=0$, we get
$$
J_d\le 1+p^{1/18}d^{101/126}\Bigl(\frac{1}{p}\sum_{a=1}^{p-1}\Bigl|\sum_{1\le z\le (p-1)/d}e_p(adz)\Bigr|\Bigr)\lesssim p^{1/18}d^{101/126}.
$$
Using Lemma~\ref{lem:Vin}, we get the following bound for the
double:
$$
\sum_{a=1}^{p-1}\Bigl|\sum_{1\le z\le (p-1)/d}e_p(adz)\Bigr|=\sum_{b=1}^{p-1}\Bigl|\sum_{1\le z\le (p-1)/d}e_p(bz)\Bigr|\lesssim p.
$$
Therefore
$$
J_d\lesssim p^{1/18}d^{101/126}.
$$
Comparing this estimate with~\eqref{eqn: J_d 1} we obtain
$$
J_d\lesssim p^{27/82}.
$$
Incorporating this in~\eqref{eqn:R lesssim Jd}, we get the desired
result.

\section{Proof of Theorem~\ref{thm:Main2}}
In view of Lemma~\ref{lem:BBSh}, it suffices to deal with the case
$t<p^{1/3}$.

Since $\lambda^t\equiv 1\pmod p$, it follows
from~\eqref{eqn:MainCongr} that
$$
\sum_{\substack{1\le \lambda\le p-1\\ {\rm ord} \,\lambda =t}}J(p; \lambda)\le \#\{x\in \mathbb{N};\quad x^{tx}\equiv 1\pmod p, \quad x\le p-1\}
$$
Hence, denoting $d=\gcd(x,(p-1)/t)$ and using Lemma \ref{lem:easy}
we obtain that
$$
\sum_{\substack{1\le \lambda\le p-1\\ {\rm ord} \,\lambda =t}}J(p; \lambda) \le \sum_{d | (p-1)/t} T_d,
$$
where $T_d$ is the number of solutions of the congruence
$$
z^{dt}\equiv (d^{dt})^{*}\pmod p,\quad z\in\N, \quad z\le (p-1)/d.
$$
By the trivial estimate $T_d\le p/d$ we have
$$
\sum_{\substack {d|p-1\\ d>p^{2/3}}}T_d\le \sum_{d|p-1} p^{1/3}\lesssim p^{1/3}.
$$
Furthermore, applying Lemma~\ref{lem:CongrMoment} with $k=2$, we get
$$
\sum_{\substack {d|p-1\\ p^{1/3}<d<p^{2/3}}} T_d\le \sum_{\substack {d|p-1\\ p^{1/3}<d<p^{2/3}}} \Bigl(d^{1/2}+(p/d)^{1/2}\Bigr)t^{1/2}\lesssim p^{1/3}t^{1/2}.
$$
Therefore,
\begin{equation}
\label{eqn:preultimateTh2}
\sum_{\substack{1\le \lambda\le p-1\\ {\rm ord} \,\lambda =t}}J(p; \lambda) \le p^{1/3}t^{1/2} + \sum_{\substack {d | (p-1)/t\\ d<p^{1/3}}} T_d.
\end{equation}
Recall that $t<p^{1/3}$, thus $dt| p-1$ and $dt<p^{2/3}$.

Let $H_{dt}$ be the subgroup of $\F_p^*$ of order $dt$. Since  $T_d$
is the number of solutions of the congruence
$$
(dz)^{dt}\equiv 1\pmod p;\quad z\in \N, \quad z\le (p-1)/d,
$$
it follows that
$$
T_d= \#\{z\in \N; \quad z\le (p-1)/d,\quad dz\pmod p\in H_{dt}\}.
$$
Therefore,
$$
T_d=\frac{1}{p}\sum_{a=0}^{p-1}\sum_{1\le z\le (p-1)/d}\sum_{h\in H_{dt}}e_p(a(dz-h)).
$$
Separating the term corresponding to $a=0$ and using
Lemma~\ref{lem:Shk} for $a\not=0$ (with $d$ replaced by $dt$), we
get
$$
T_d\le t+p^{1/6}d^{1/2}t^{1/2}\Bigl(\frac{1}{p}\sum_{a=1}^{p-1}\Bigl|\sum_{1\le z\le (p-1)/d}e_p(adz)\Bigr|\Bigr).
$$
Applying Lemma~\ref{lem:Vin} to the double sum, as in the proof of
Theorem~\ref{thm:Main1}, we obtain for $d<p^{1/3}$ the bound
$$
T_d\lesssim t+p^{1/6}d^{1/2}t^{1/2}\lesssim t+p^{1/3}t^{1/2}.
$$
Thus,
$$
\sum_{\substack {d | (p-1)/t\\ d<p^{1/3}}} T_d\le\sum_{d | p-1}(t+p^{1/3} t^{1/2})\lesssim t+ p^{1/3}t^{1/2}.
$$
Putting this into~\eqref{eqn:preultimateTh2}, we conclude the proof.

\section{Proof of Theorem~\ref{thm:Main3}}

We follow the arguments of~\cite{BBSh} with some modifications. We
have
$$
I(p)=\sum_{\lambda=1}^{p-1}J(p;\lambda)^2=\sum_{t|p-1}\, \sum_{\substack{1\le\lambda\le p-1\\ {\rm ord}\,\lambda=t}}J(p;\lambda)^2.
$$
It then follows that for some fixed order $t|p-1$ we have
$$
I(p)\lesssim \sum_{\substack{1\le\lambda\le p-1\\ {\rm ord}\,\lambda=t}}J(p;\lambda)^2.
$$
We can split the range of $J(p;\lambda)$ into $O(\log p)$ dyadic
intervals. Then, for some $1\le M\le p$, we have
\begin{equation}
\label{eqn:I(p)<AMM}
I(p)\lesssim |\mathcal{A}|M^2,
\end{equation}
where $|\mathcal{A}|$ is the cardinality of the set
$$
\mathcal{A}=\{1\le \lambda\le p-1; \quad {\rm ord}\,\lambda=t, \quad M\le J(p;\lambda)<2M\}.
$$
From Lemma~\ref{lem:BBSh1} we have
\begin{equation}
\label{eqn:M<pt}
M\lesssim pt^{-1/12}.
\end{equation}
On the other hand, by Lemma~\ref{lem:BBSh} we also have
$$
|\cA|M\lesssim \sum_{\substack{\lambda\in\mathcal{A}}}J(p;\lambda)\lesssim \sum_{\substack{1\le\lambda\le p-1\\ {\rm ord}\,\lambda=t}}J(p;\lambda)\lesssim t+p^{1/2}.
$$
If $t<p^{1/2}$, then using~\eqref{eqn:I(p)<AMM} we get
$$
I(p)\lesssim |\cA|M^2\lesssim (|\cA|M)^2 \lesssim p,
$$
and the result follows. If $t>p^{1/2}$, then we get $|\cA|M\lesssim
t$. Therefore, using~\eqref{eqn:I(p)<AMM} and~\eqref{eqn:M<pt} we
get
$$
I(p)\lesssim |\cA|M^2\lesssim t (pt^{-1/12})=pt^{11/12}\lesssim p^{23/12}.
$$
This proves Theorem~\ref{thm:Main3}.

Address of the authors:\\

J. Cilleruelo, Instituto de Ciencias Matem\'{a}ticas
(CSIC-UAM-UC3M-UCM) and Departamento de Matem\'{a}ticas, Universidad
Aut\'{o}noma de Madrid, Madrid-28049, Spain.

Email:{\tt franciscojavier.cilleruelo@uam.es}.

\vspace{1cm}

M.~Z.~Garaev, Centro de Ciencias Matem\'{a}ticas,  Universidad
Nacional Aut\'onoma de M\'{e}xico, C.P. 58089, Morelia,
Michoac\'{a}n, M\'{e}xico,

Email:{\tt garaev@matmor.unam.mx}

\end{document}